\documentclass[11pt]{amsart}

\usepackage{color,graphicx,enumerate,wrapfig,amssymb,mathabx}
\usepackage{subfigure}
\usepackage{eucal}

\usepackage{setspace}

\usepackage{hyperref}

\hypersetup{
    bookmarks=true,         
    unicode=true,          
    pdftoolbar=true,        
    pdfmenubar=true,        
    pdffitwindow=false,     
    pdfstartview={FitH},    
    pdftitle={My title},    
    pdfauthor={Author},     
    pdfsubject={Subject},   
    pdfcreator={Creator},   
    pdfproducer={Producer}, 
    pdfkeywords={keyword1} {key2} {key3}, 
    pdfnewwindow=true,      
    colorlinks=true,       
    linkcolor=blue,          
    citecolor=blue,        
    filecolor=magenta,      
    urlcolor=cyan           
}

\renewenvironment{proof}[1][\proofname ]{{\noindent \bfseries #1. }}{\qed \bigskip } 
\def\subjclass#1{{\renewcommand{\thefootnote}{}%
\footnote{\emph{Mathematics Subject Classification (2010):} #1}}}

\newcommand{\R}{{\mathbb R}}
\newcommand{\Z}{{\mathbb Z}}

\newcommand{\e}{\varepsilon}

\newcommand{\supp}{\operatorname{supp}}

\def\subjclass#1{{\renewcommand{\thefootnote}{}%
\footnote{\emph{Mathematics Subject Classification (2010):} #1}}}

\newtheorem{theorem}{Theorem}[section]

\newtheorem{lem}[theorem]{Lemma}

\numberwithin{equation}{section}

\begin{document}

\title[On localization of Schr\"{o}dinger means]{On localization of Schr\"{o}dinger means}

\author{Per Sj\"{o}lin}

\address{Department of Mathematics, KTH Royal Institute of Technology,
  100~44  Stockholm, Sweden}
\email{persj@kth.se}


    \begin{abstract}
 Localization properties for Schr\"{o}dinger means are studied in dimension higher than one.  
    \end{abstract}

\maketitle

\section{Introduction}\label{sec-intro}

Let $f$ belong to the Schwartz class $\mathcal{S}(\R^n)$ where $n\geq 1$. We define
the Fourier transform $\widehat{f}$ by setting
$$
\widehat{f}(\xi) = \int\limits_{\R^n} e^{-i \xi \cdot x} f(x) dx , \ \ \xi\in\R^n.
$$
For $f\in \mathcal{S}(\R^n)$ we also set
\begin{equation}\label{eq-1}
S_t f(x) = \int\limits_{\R^n} e^{i \xi \cdot x} e^{i t |\xi|^2 } \widehat{f}(\xi) d\xi, \ \ x\in \R^n, \ t\in \R.
\end{equation}
If we set $u(x,t)  = S_t f(x) / (2\pi)^n$, then $u(x,0)= f(x)$
and $u$ satisfies the Schr\"{o}dinger equation 
$i \partial u / \partial t = \Delta u$.

It is well-known that $e^{i|\xi|^2}$ has the Fourier transform $K(x) = c e^{-i|x|^2/4}$,
where $c$ denotes a constant, and $e^{i t |\xi|^2}$ has the Fourier transform
$$
K_t(x) = \frac{1}{t^{n/2}} K\left( \frac{x}{t^{1/2}}  \right), \ x\in \R^n, \ t>0.
$$
One has $S_t f(x) = K_t \star f(x) $ for $f\in \mathcal{S}(\R^n) $ and $t>0$,
and we set
\begin{equation}\label{eq-2}
S_t f(x) = K_t \star f(x) , \ \ t>0,
\end{equation}
for $f\in L^1(\R^n)$. For $f\in L^2(\R^n)$ we define $S_t f$ by formula \eqref{eq-1}.

We introduce Sobolev spaces $H_s= H_s(\R^n)$ by setting
$$
H_s = \{  f\in \mathcal{S}': \ || f||_{H_s}<\infty \}, \ \ s\in \R,
$$
where
$$
|| f ||_{H_s} = \left(  \int_{\R^n} ( 1+|\xi|^2 )^s  | \widehat{f}(\xi) |^2 d \xi  \right)^{1/2}
$$
In the case $n=1$ it is well-known (see Carleson \cite{C} and Dahlberg and Kenig \cite{DC})
that
$$
\lim\limits_{t\to 0} \frac{1}{2\pi} S_t f(x) = f(x)
$$
almost everywhere if $f\in H_{1/4}$. Also it is known that $H_{1/4}$ cannot be replaced
by $H_s$ if $s<1/4$. In the case $n\geq 2$ Sj\"{o}lin \cite{S1} and Vega \cite{V} proved
independently that
$$
\lim\limits_{t\to 0} \frac{1}{(2\pi )^n} S_t f(x) = f(x)
$$
almost everywhere if $f\in H_s$ and $s> 1/2$. This result was improved by Bourgain \cite{B1},
who proved that $f\in H_s(\R^n)$, $s> 1/2 - 1/4n$, is sufficient for convergence almost everywhere.

In the case $n=2$, Du, Guth, and Li \cite{DGL} have recently proved that the condition $s>1/3$ is sufficient.
On the other hand Bourgain \cite{B2} has proved that $s\geq n/2(n+1)$ is necessary for convergence for all $n\geq 2$.

We shall here study localization of Schr\"{o}dinger means and shall first state a result
on localization everywhere (see Sj\"{o}lin \cite{S2}).

\smallskip

\noindent \textbf{Theorem A.} \emph{Assume $n\geq 1$. If $f\in H_{n/2}(\R^n)$ and $f$ has compact support then}
$$
\lim\limits_{t\to 0} S_t f(x) = 0
$$
\emph{for every $x\in \R^n\setminus (\mathrm{supp} f)$}.

\bigskip

It is also proved in \cite{S2} that this result is sharp in the sense that $H_{n/2}$
cannot be replaced by $H_s $ with $s<n/2$. 

We say that one has localization almost everywhere for functions in $H_s$ if for every $f\in H_s$ one has
$$
\lim\limits_{t\to 0} S_t f(x) =0
$$
almost everywhere in $\R^n \setminus (\mathrm{supp} f)$.

In the case $n=1$ Sj\"{o}lin and Soria proved that there is no localization almost everywhere for
functions in $H_s$ if $s<1/4$ (see Sj\"{o}lin \cite{S3}). In fact they proved that there exist
two disjoint compact intervals $I$ and $J$ in $\R$ and a function $f$ which belongs to $H_s$ for every
$s<1/4$, with the properties that $\mathrm{supp} f\subset I$ and for every $x\in J$ one does not have
$$
\lim\limits_{t\to 0} S_t f(x) =0.
$$

In the case $n\geq 2$ Sj\"{o}lin and Soria also proved that one does not have localization
almost everywhere for functions in $H_s(\R^n)$ if $s<1/4$.

We shall here improve this result and prove that there is no localization almost everywhere for 
functions in $H_s(\R^n)$ if $n\geq 2$ and $s<n/2(n+1)$. In fact we shall prove the following theorem.

\begin{theorem}\label{thm-1}
If $n\geq 2$ and $s< n/2(n+1)$ there exist a function $f$ in $H_s(\R^n) \cap L^1(\R^n)$ and a set
$F$ with positive Lebesgue measure such that 
$F\subset \R^n \setminus (\mathrm{supp} f)$ and for every $x\in F$ one does not have 
$\lim\limits_{t\to 0} S_t f(x) =0$.
\end{theorem}

\bigskip

To prove this result we shall combine the method in \cite{S3} with an estimate of Bourgain \cite{B2}.

If $A$ and $B$ are numbers we write $A\lesssim B$ if there is a positive constant $C$ such that $A \leq C B$.
If $A\lesssim B$ and $B\lesssim A$ we write $A \sim B$.

We introduce the inverse Fourier transform by setting
$$
\widecheck{f} (x) = (2\pi)^{-n} \int\limits_{\R^n} e^{i\xi \cdot x} f(\xi) d\xi , \ \ x\in \R^n,
$$
for $f\in L^1(\R^n)$.

Also $B(x; r)$ denotes a ball with center $x$ and radius $r$.

\section{Proof of the theorem}

We start by taking $v_1$ such that $0<v_1<1$ and set $v_k = \e_k v_{k-1}^\mu$
for $k=2,3,4,...,$ where $\e_k = 2^{-k}$ and $\mu = \max( n, 2+n/4 )$.
Then one has $v_k<2^{-k}$ for $k\geq 2$ and $(v_k)_1^\infty$ is a decreasing
sequence tending to zero.

We then choose $g\in \mathcal{S}(\R)$ such that $\supp \widecheck{g} \subset (-1,1)$,
$\widecheck{g}(x_1) = 1$ for $|x_1| < 1/2 $, and set
$f_v (x_1) = e^{-i x_1/v^2} \widecheck{g}(x_1/v)$, $0<v<1$, $x_1 \in \R$.

The functions $f_v$ were used in Sj\"{o}lin \cite{S3} to study the localization
problem in the case $n=1$. Also let $\Phi \in \mathcal{S}(\R^{n-1})$ have $\supp \widehat{\Phi} \subset B(0; 1)$
and $\Phi(0)=1$. We then take $R=1/v^2$ and set 
$$
G_v(x') = R^{- (n-1)/4 } \Phi(x') \prod\limits_{j=2}^n \left( \sum\limits_{R/2D < l_j < R/D}
e^{i D l_j x_j}  \right), \ \ 0<v<1,
$$
where $x'=(x_2,...,x_n)\in \R^{n-1}$, $l=(l_2,...,l_n) \in \Z^{n-1}$, and
$D= R^{(n+2)/2(n+1)}$. We may also assume that $\Phi(x') = \psi(x_2)...\psi(x_n)$
for some $\psi \in \mathcal{S}(\R)$.

We then set
$$
h_v(x) = h_v(x_1, x') = f_v(x_1) G_v(x') , \ \ 0<v<1.
$$
In \cite{B2} Bourgain studies functions similar to $h_v$.
However, in \cite{B2} our function $\widecheck{g}$ is replaced by a function
$\varphi$ with the property that $\widehat{ \varphi }$ has compact support.
In our argument it will be important that $\widecheck{g}$ has compact support so that
$$
\supp f_v \subset (-v, v).
$$
We then observe that
$$
S_t h_v(x_1, x') = S_t f_v(x_1) S_t G_v(x').
$$
It is proved in Bourgain \cite{B2}, p.394, that if one assumes $|x|<c$ and $|t|<c/R$
and sets
\begin{equation}\label{eq-3}
t=\frac{x_1}{2R} + \tau
\end{equation}
with $|\tau|< R^{-3/2}/10$, then
\begin{equation}\label{eq-4}
|S_t f_v (x_1)| \gtrsim |\widecheck{g} ( R^{1/2} x_1 - 2t R^{3/2} ) | = |\widecheck{g} (2\tau R^{3/2})| \geq c_0.
\end{equation}
We then take $v=v_k$ for $k=1,2,3,...$, and apply an estimate in \cite{B2}, p.395, namely that there exists
a set $E_k\subset B(0; 1)$ such that for every $x\in E_k$ there exists $t=t_k(x)$
such that
$$
|S_{t_k(x)} G_{v_k }(x')| \gtrsim R^{-(n-1)/4} \prod\limits_{j=2}^n \left| \sum\limits_{l_j}
e^{i D l_j x_j} e^{i D^2 l_j^2 t}  \right| \geq c_0.
$$
Also one has $m E_k \geq c_1 >0$, where $m$ denotes Lebesgue measure.

We then choose $\delta>0$ so small that if
$F_k = E_k \cap \{x;  \ |x_1|>\delta/2 \}$ then one has
$m F_k \geq c_1/2 = c_2 $ for $k=1,2,3,...$ . We may assume that $\delta<1$.

We then set $F= \bigcap\limits_{k=1}^\infty \left( \bigcup\limits_{j=k}^\infty F_j  \right)$
so that $F$ is the set of all $x$ which belong to infinitely many $F_k$. 
Since the sets  $\bigcup\limits_{j=k}^\infty F_j$ decrease as $k$ increases,
one obtains $m F \geq c_0$.

It also follows from \eqref{eq-4} that
$|S_{t_k(x)} f_{v_k}(x)  | \geq c_0  $  for $x\in  F_k$. From \eqref{eq-3} one also concludes that
\begin{equation}\label{eq-5}
|t_k(x) | \sim \frac{1}{R_k} = v_k^2 
\end{equation}
where $R_k = 1/v_k^2$.

We now choose $K$ so large that $v_K< \delta/4$ and set
$$
h(x) = \sum\limits_{k=K}^\infty h_{v_k} (x) =  \sum\limits_{k=K}^\infty f_{v_k } (x_1) G_{v_k}(x').
$$
One has $F\subset \R^n \setminus (\supp h)$ since $\supp h \subset \{x; \ |x_1|<\delta/4 \}$.

If $x=(x_1, x') \in F$ there exists a sequence $(k_j)_{j=1}^\infty$ such that $x\in F_{k_j}$ and
$$
| S_{t_{k_j}(x) } G_{v_{k_j}} (x') | \geq c_0 \ \ \text{ and } \ \ |S_{t_{k_j}(x)} f_{v_{k_j}} (x_1) | \geq c_0.
$$
We shall prove that
$$
| S_{ t_{k_j} (x) } h(x) | \geq c_0
$$
for $j$ large, and also that $h\in H_s \cap L^1$ for $s<n/2(n+1)$. 
It follows that one does not have localization almost everywhere in $H_s$ if $s<n/2(n+1)$.

We shall first estimate $|| h_v ||_{H_s} $. We begin by studying $f_v$ and $G_v$.
According to Sj\"{o}lin \cite{S3}, p. 143, one has
$$
\widehat{f}_v (\xi_1) = v g(v \xi_1 + 1/v) = R^{-1/2} g(R^{-1/2} \xi_1 + R^{1/2})
$$
and for $|\xi_1| \geq A R$ we have 
$$
R^{-1/2} |\xi_1| \geq R^{1/2} A,
$$
where $A$ is a large constant. It follows that $|\xi_1|\geq A R$ implies 
$$
|R^{-1/2} \xi_1 + R^{1/2} | \geq B R^{-1/2} |\xi_1|
$$
and
$$
|\xi_1 | \geq |\xi_1|^{1/2} A_1 R^{1/2}.
$$
Hence
$$
|R^{-1/2}\xi_1  +R^{1/2} |  \geq B_1 |\xi_1|^{1/2}
$$
and
$$
| \widehat{f}_v (\xi_1) | \leq C |\xi|^{-N} \ \  \text{ for } \ \ |\xi_1| \geq A R,
$$
where $N$ is large. It follows that
\begin{equation}\label{eq-6}
\int\limits_{|\xi_1| \geq A R } | \widehat{f}_v (\xi_1)   |^2 ( 1+\xi_1^2 )^s d\xi_1 \leq C R^{-N}
\end{equation}
and it is also easy to see that
\begin{equation}\label{eq-7}
|| f_v ||_2 \sim v^{1/2}.
\end{equation}

We have
$$
\widehat{G}_v (\xi') = R^{-(n-1)/4} \sum\limits_l \widehat{\Phi} (\xi' - D l)
$$
and it follows that $\supp \widehat{G}_v \subset \{ \xi' ; \ |\xi'|\sim R  \}$
and
$$
| \widehat{G}_v (\xi') |^2 = R^{-(n-1)/2} \sum\limits_{l} | \widehat{\Phi}(\xi' - D l) |^2.
$$
Integrating we obtain
$$
|| \widehat{G}_v ||_2^2 = R^{-(n-1)/2} \sum\limits_l || \widehat{\Phi}||_2^2 \sim R^{-(n-1)/2} 
\left( \frac RD \right)^{n-1}.
$$
We have $D=R^{(n+2)/2(n+1)}$ and $R/D = R^{n/2(n+1)}$ and hence
$$
|| G_v||_2^2 \sim R^{-(n-1)/2} R^{n(n-1)/2(n+1)} = R^{-(n-1)/2(n+1)}
$$
and
\begin{equation}\label{eq-8}
|| G_v||_2 \sim R^{-(n-1)/4(n+1)} .
\end{equation}

For $s>0$ one obtains
\begin{multline*}
|| h_v||_{H_s}^2 \sim \int | \widehat{h}_v (\xi) |^2 | \xi |^{2s} d\xi = 
\int\limits_{|\xi'|\sim R} | \widehat{f}_v (\xi_1) |^2 | \widehat{G}_v (\xi') |^2 |\xi|^{2s} d\xi \lesssim \\
\iint_{\substack{|\xi' | \sim R \\ |\xi_1 |\leq A R }}
|\widehat{f}_v(\xi_1)|^2 | \widehat{G}_v (\xi') |^2 R^{2s} d\xi_1 d \xi' + 
\iint_{\substack{|\xi' | \sim R \\ |\xi_1 |\geq A R }}
|\widehat{f}_v(\xi_1)|^2 | \widehat{G}_v (\xi') |^2 |\xi_1|^{2s} d\xi_1 d \xi' = I_1 +I_2.
\end{multline*}
It follows that
\begin{multline*}
I_1 \lesssim R^{2s} \left(  \int_{|\xi_1|\leq A R} |\widehat{f}_v(\xi_1)|^2  d\xi_1  \right)
\left(\int |\widehat{G}_v (\xi')| d\xi'  \right) \lesssim \\
R^{2s} ||  f_v ||_2^2 || G_v||_2^2 \lesssim R^{2s} R^{-1/2} R^{-(n-1)/2(n+1)} = R^{2s -n/(n+1)}.
\end{multline*}

From \eqref{eq-6} and \eqref{eq-8} one also obtains
$$
I_2 \lesssim R^{-N}
$$
and hence
$$
I_1 + I_2 \lesssim R^{2s - n/(n+1)}.
$$

It follows that
$$
|| h_v || _{H_s} \lesssim R^{s-n/2(n+1)} = v^{2 (n/2(n+1) -s)  }.
$$

Since $v_k < \e_k$ and
$$
|| h_k||_{H_s} \leq \sum\limits_{K}^\infty || h_{v_k} ||_{H_s} 
\lesssim \sum\limits_{K}^\infty v_k^{ 2 ( n/2(n+1) -s ) } <\infty,
$$
it follows that $h\in H_s$ if $s<n/2(n+1)$.

We also need some estimates for $S_t f_v$ and $S_t G_v$.
In Sj\"{o}lin \cite{S3} (see Lemmas 3 and 4) it is proved that
\begin{equation}\label{eq-9}
|S_t f_v (x_1) | \lesssim \frac{v}{|t|^{1/2}}
\end{equation}
and
\begin{equation}\label{eq-10}
|S_t f_v(x_1)| \lesssim \frac{|t|}{v^4}
\end{equation}
for $0<v<\delta/4$, $|x_1|\geq \delta/2$, and $0<|t|<1$.
Actually it is assumed in \cite{S3} that $t>0$ but the same proofs work for $t<0$.

We also have the following estimates for $S_t G_v$.

\begin{lem}\label{Lem-1}
For $0<v<\delta/4$, $0<|t|<1$, and $x' \in \R^{n-1}$ one has
\begin{equation}\label{eq-11}
|S_t G_v(x')| \lesssim v^{(n-1)/2} ( \log 1/v )^{n-1} |t|^{-(n-1)/2}
\end{equation}
and
\begin{equation}\label{eq-12}
|S_t G_v(x')| \lesssim v^{-(n-1)^2/2(n+1)}.
\end{equation}
\end{lem}

\begin{proof}
Choose the integer $p$ so that $|4p - R/D| \leq 4$.
Then one has
$$
\sum\limits_{R/2D < l< R/D} e^{i D l x_j}  = \sum\limits_{2p}^{4p} e^{i D l x_j} + \mathrm{O} (1)
$$
and
$$
\left| \sum\limits_{2p}^{4p} e^{i D l x_j}  \right| = \left| \sum\limits_{-p}^p e^{i D (l+3p) x_j}  \right| =
\left|\sum\limits_{-p}^p e^{i l D x_j}  \right| = D_p(Dx_j),
$$
where $D_p$ denotes the Dirichlet kernel. Setting $y= D x_j$ one obtains
$$
\int\limits_{a}^{a+1/D} |D_p (D x_j)| d x_j = \int\limits_{D a}^{ Da +1} |D_p (y)| dy \frac{1}{D} \lesssim
\frac 1D \log p \sim \frac 1D \log R
$$
for every $a\in \R$. It follows that
$$
\int\limits_{a}^{a+1} | D_p (D x_j)| dx_j \lesssim \log R
$$
for every $a\in \R$.

Letting $Q$ denote the cube $[0,1]^{n-1}$ we obtain
\begin{multline*}
|| G_v||_1 = \int\limits_{\R^{n-1}} |G_v(x') | dx' = \sum\limits_{m\in \Z^{n-1} } \int\limits_{m+Q} |G_v(x')| dx' \lesssim \\
R^{-(n-1)/4} \sum\limits_{m} \frac{1}{1+|m|^N} \int\limits_{m+Q} 
\left( \prod_{j=2}^n  \left| \sum\limits_{l_j} e^{i l_j D x_j} \right|  \right) dx' \lesssim \\
R^{-(n-1)/4} \sum\limits_m \frac{1}{1+|m|^N}  \prod\limits_{j=2}^n 
\left( \int\limits_{m_j}^{m_j+1}  \left| \sum\limits_{l_j} e^{i l_j D x_j}  \right| dx_j  \right) \lesssim \\
R^{-(n-1)/4} \sum\limits_{m} \frac{1}{1+|m|^N} (\log R)^{n-1}
\end{multline*}
and hence
$$
|| G_v||_1 \lesssim R^{-(n-1)/4} (\log R)^{n-1} \sim v^{(n-1)/2} (\log 1/v)^{n-1}.
$$
We have
$$
S_t G_v(x') = \int K_t(x' - y') G_v(y') dy'
$$
where 
$$
|K_t (y')| \lesssim |t|^{-(n-1)/2}
$$
and it follows that
$$
|S_t G_v(x')| \lesssim |t|^{-(n-1)/2} || G_v||_1
$$
and hence we obtain \eqref{eq-11}.

We also have
$$
S_t G_v(x') = \int\limits_{\R^{n-1}} e^{i \xi' \cdot x'} e^{i t |\xi'|^2} \widehat{G}_v (\xi') d\xi' = 
R^{-(n-1)/4} \sum\limits_{l} \int\limits_{\R^{n-1}} e^{i \xi' \cdot x'} e^{i t |\xi'|^2} \widehat{\Phi}(\xi' - D l) d\xi'
$$
and we get
$$
|S_t G_v(x')| \lesssim R^{-(n-1)/4} \sum\limits_{l} ||\widehat{\Phi}||_1 \lesssim R^{-(n-1)/4} (R/D)^{n-1} =
R^{(n-1)^2 /4(n+1)} = v^{-(n-1)^2/2(n+1)}
$$
and the proof of Lemma \ref{Lem-1} is complete.
\end{proof}

Multiplying \eqref{eq-9} and \eqref{eq-11} one obtains
\begin{equation}\label{eq-13}
|S_t h_v(x)| \lesssim v^{(n+1)/2} (\log 1/v)^{n-1} |t|^{-n/2} \lesssim \frac{v}{|t|^{n/2}} = \frac{v}{|t|^\gamma}
\end{equation}
and combining \eqref{eq-10} and \eqref{eq-12} one gets
\begin{equation}\label{eq-14}
|S_t h_v(x)| \lesssim \frac{|t|}{v^{4 + (n-1)^2/2(n+1) } } \lesssim \frac{|t|}{v^{4+n/2}} = \frac{|t|}{v^\beta}
\end{equation}
where $\gamma = n/2$, $\beta = 4+n/2$, $0<v<\delta/4$, $0<|t|<1$, and $|x_1|>\delta/2$.

We remark that it also follows from the above estimates that $h\in L^1(\R^n)$.
In fact it is easy to see that $||f_v||_1 \sim v$ and we have proved that
$$
|| G_v||_1 \lesssim v^{(n-1)/2} (\log 1/v)^{n-1} \lesssim v^{1/4}
$$
and hence $||h_v||_1 \lesssim v^{5/4}$. It follows that
$$
||h||_1 \leq \sum\limits_{K}^\infty ||h_{v_k}||_1 <\infty.
$$
We shall now finish the proof of the following result.

\begin{theorem}
For $x\in F$ one has
$$
|S_{t_{k_j}(x)} h(x)| \geq c_0 >0
$$
for $j$ large. It follows that there is no localization almost everywhere of 
Schr\"{o}dinger means for functions in $H_s(\R^n)$ if $s<n/2(n+1)$.
\end{theorem}
\begin{proof}
We have
$$
|S_{t_{k_j}(x)} h(x) | \geq |S_{t_{k_j}(x)} h_{v_{k_j}}(x) | - \left| \sum\limits_{i\neq k_j} S_{t_{k_j}(x)} h_{v_i} (x)  \right|.
$$
The first term on the right hand side is larger than a positive number $c_0$ and it suffices to prove that
the second term is small. For simplicity we write $k$ instead of $k_j$ in the following formulas.

We have $0<v_i \leq v_{i-1}/2$ and $0<v_i \leq 2^{-i}$ and it follows that
\begin{equation}\label{eq-15}
\sum\limits_{i=k+1}^\infty v_i \leq 2 v_{k+1} 
\end{equation}
and one also has
\begin{equation}\label{eq-16}
\sum\limits_{i=K}^{k-1} \frac{1}{v_i^\beta} \lesssim \frac{1}{v_{k-1}^\beta}
\end{equation}
since $1/v_{i-1} \leq 1/2v_i$ implies $1/v_{i-1}^\beta \leq \frac{1}{2^\beta} \frac{1}{v_i^\beta}$.

For $i\geq k+1$ the inequality \eqref{eq-13} and the formula \eqref{eq-5} give
$$
|S_{t_k(x)} h_{v_i}(x) | \lesssim \frac{v_i}{(v_k^2)^\gamma} = \frac{v_i}{v_k^{2\gamma}}
$$
and invoking \eqref{eq-15} we obtain
$$
\sum\limits_{i=k+1}^\infty |S_{t_k(x)} h_{v_i}(x)| \lesssim \frac{v_{k+1}}{v^{2\gamma}_k} \lesssim \e_{k+1},
$$
since $\mu \geq 2\gamma$ and $v_{k+1} = \e_{k+1}v_k^\mu \leq \e_{k+1} v_k^{2\gamma}$.
For $K\leq i \leq k-1$ the inequality \eqref{eq-14} gives 
$$
|S_{t_k(x)} h_{v_i}(x)| \lesssim \frac{|t_k(x)|}{v_i^\beta} \lesssim \frac{v_k^2}{v_i^\beta}
$$
and invoking \eqref{eq-16} one obtains
\begin{equation}\label{eq-17}
\sum\limits_{i=K}^{k-1} | S_{t_k(x)} h_{v_i}(x) | \lesssim v_k^2 \sum\limits_{K}^{k-1} \frac{1}{v_i^\beta}
\lesssim \frac{v_k^2}{v_{k-1}^\beta}.
\end{equation}
Since $\mu \geq \beta/2$ we get $v_k \leq \e_k v_{k-1}^{\beta/2}$ and
$v_k^2 \leq \e_k^2 v_{k-1}^\beta$. Hence the sum in \eqref{eq-17} is majorized by $C \e_k^2$.
Since $\e_k \to 0$ as $k\to \infty$ the proof of the theorem is complete.

\end{proof}

\end{document}